\documentclass[11pt]{amsart}
\usepackage{amssymb}
\usepackage{mathtools}
\usepackage{mathrsfs,enumerate,color,graphicx,subfigure}
\usepackage[bookmarks,pagebackref]{hyperref}
\usepackage[backrefs,alphabetic,initials]{amsrefs}
\usepackage[linesnumbered,ruled,vlined]{algorithm2e}

\newtheorem{theorem}{Theorem}[section]
\newtheorem{lemma}[theorem]{Lemma}
\newtheorem{proposition}[theorem]{Proposition}

\theoremstyle{definition}
\newtheorem{definition}[theorem]{Definition}
\newtheorem{example}[theorem]{Example}

\theoremstyle{remark}
\newtheorem{remark}[theorem]{Remark}

\theoremstyle{definition}
\newtheorem{problem}{Problem}

\numberwithin{equation}{section}

\newcommand{\bp}{\mathsf{bp}}
\newcommand{\C}{\mathbb{C}}
\newcommand{\Cs}{\mathscr{C}}
\newcommand{\cb}{\mathbf{c}}
\newcommand{\co}{\coloneqq}
\newcommand{\db}{\mathbf{d}}
\newcommand{\de}{\mathrm{der}}
\newcommand{\I}{\mathrm{init}}
\newcommand{\kk}{\Bbbk}
\newcommand{\K}{\mathbb{K}}
\newcommand{\F}{\mathbb{F}}
\newcommand{\mv}{\mathrm{mvar}}
\newcommand{\N}{\mathbb{N}}
\newcommand{\p}{\mathsf{Projection}}
\newcommand{\Q}{\mathbb{Q}}
\newcommand{\Qc}{\mathcal{Q}}
\newcommand{\R}{\mathbb{R}}
\newcommand{\Rc}{\mathcal{R}}
\newcommand{\Rs}{\mathscr{R}}
\newcommand{\RT}{\mathsf{RealTriangularize}}
\newcommand{\re}{\mathrm{res}}
\newcommand{\s}{\mathrm{span}}
\newcommand{\sa}{\mathrm{sat}}
\newcommand{\Sf}{\mathfrak{S}}
\newcommand{\T}{\mathsf{Triangularize}}
\newcommand{\U}{\mathbf{u}}
\newcommand{\y}{\mathbf{y}}
\newcommand{\x}{\mathbf{x}}

\usepackage[left=3cm,right=3cm,top=2.5cm,bottom=2.5cm]{geometry}

\title[Testing isomorphism of complex and real Lie algebras]{Testing isomorphism of complex and real Lie algebras}

\author{Tuan A. Nguyen \and Vu A. Le \and Thieu N. Vo}

\address{Tuan A. Nguyen, Faculty of Political Science and Pedagogy, University of Physical Education and Sports, Ho Chi Minh City, Vietnam.}
\email{natuan@upes.edu.vn}

\address{Vu A. Le, Department of Economic Mathematics, University of Economics and Law, Vietnam National University - Ho Chi Minh City, Vietnam.}
\email{vula@uel.edu.vn}

\address{Thieu N. Vo, Fractional Calculus, Optimization and Algebra Research Group, Faculty of Mathematics and Statistics, Ton Duc Thang University, Ho Chi Minh City, Vietnam}
\email{vongocthieu@tdtu.edu.vn}


\begin{document}

\begin{abstract}
	In this paper, we give algorithms for determining the existence of isomorphism between two
	finite-dimensional Lie algebras and compute such an isomorphism in the affirrmative case. 
	We also provide algorithms for determining algebraic relations of parameters in order to decide 
	whether two parameterized Lie algebras are isomorphic. All of the considered Lie algebras are considred over 
	a field $\F$, where $\F=\C$ or $\F=\R$. 
	Several illustrative examples are given to show the applicability and the effectiveness of the proposed algorithms.
\end{abstract}

\maketitle

\section{Introduction}

In this paper, we will consider a computer-based approach for solving the following two problems:

\begin{problem}\label{Problem1}
Given two  $\F$-Lie algebras $L$ and $L'$ of same dimension. Deciding whether L and
L0 are isomorphic or not, and determine an isomorphism in the affirmative case.
\end{problem}

\begin{problem}\label{Problem2}
	Given two $\F$-Lie algebras $L(\cb)$ and $L'(\db)$ depending on $r$-tuple $\cb = (c_1, \dotsc, c_r) \in \F^r$ 
	and $s$-tuple $\db = (d_1, \dotsc, d_s) \in \F^s$, respectively. 
	Find conditions of parameters $\cb$ and $\db$ such that $L(\cb)$ and $L'(\db)$ are isomorphic.
\end{problem}

By definition, an \emph{$\F$-Lie algebra}, say $L$, is an $\F$-vector space endowed with a skew-symmetric bilinear map $[\cdot, \cdot] \colon L \times L \to L$ which obeys the Jacobi identity:
\[
	[[X, Y], Z] + [[Y, Z], X] + [[Z, X], Y] = 0, \quad \text{for all $X, Y, Z \in L$}.
\]
If $L$ is $n$-dimensional with a basis $\{X_1, \ldots, X_n\}$ then we have
\[
	[X_i, X_j] = \sum \limits_{k=1}^n a_{ij}^k X_k; \quad 1 \leq i < j \leq n.
\]
We call $a_{ij}^k \in \F$ the \emph{structure constants} of $L$.

An $\F$-linear isomorphism $\phi$ between two $\F$-Lie algebras $\left(L, [\cdot, \cdot]_L\right)$ and $\left(L, [\cdot, \cdot]_{L'}\right)$ is called an \emph{isomorphism} if it preserves Lie brackets, i.e.,
\begin{equation}\label{Lieisomorphism}
	\phi \left(\left[X, Y\right]_L\right) = \left[\phi(X), \phi(Y)\right]_{L'}, \quad \text{for all $X, Y \in L$}.
\end{equation}

The problem of classifying Lie algebras up to isomorphism is a fundamental problem of Lie Theory. 
Invariants (such as ideals in characteristic series, the nilradical, the center) 
contains partial informations of a Lie algebra. Invariants of isomorphic Lie algebras are the same.
However, it is still an open problem to determine a complete list of invariants such that they are strong enough to characterize a Lie algebra.
Therefore, it is impossible to decide the isomorphism between Lie algebras just by means of their invariants.

To the best of our knowledge, Gerdt and Lassner \cite{GL93} were the first authors considering Problem \ref{Problem1} 
under a view from computer algebra. In their algorithm, condition \eqref{Lieisomorphism} is transformed into a system of polynomial equations, therefore, the problem of testing Lie algebra isomorphism is reduced to the problem of testing the existence of a solution of a polynomial system.
Gr\"obner basis technique is then used to solve the latter problem.
However, since the complexity of computing Gr\"obner bases is very costly, the algorithm is impractical when the dimension pass 6.
Furthermore, it is not clear whether this algorithm is applicable for solving Problem \ref{Problem2}.

We provide new algorithms for solving Problems \ref{Problem1} and \ref{Problem2} in cases $\F = \C$ and $\R$.
Inherited from the idea by Gerdt and Lassner \cite{GL93}, we also rewrite the considered problems in terms of polynomial equations.
However, we will use the so-called \emph{triangular decomposition} instead of using Gr\"obner bases to deal with the polynomial systems.
There are two main advantages of using triangular decomposition. On the one hand, the algorithm for checking the existence 
of a solution of a polynomial system by using triangular decomposition runs much faster than that using Gr\"obner bases.
On the other hand, triangular decomposition can be used to deal with polynomial systems with parameters and over the real fields.
Details about the construction of triangular decomposition for polynomial systems with implementation in Maple were presented 
in \cite{Che11,CDMXX11,CDMMXX13,CGLMP07,CM12,CDMXX13}.

We recall necessary definitions and algorithms for triangular decomposition for polynomial systems over fields of characteristics zero 
in Sections \ref{sec2}, and for semi-algebraic system over the real field in Section \ref{sec3}.
Algorithms for the projection operator are recalled in Section \ref{sec4}.
In Section \ref{sec5}, we present algorithms for solving Problem \ref{Problem1} by using triangular decomposition.
In Section \ref{sec6}, we construct algorithms for solving Problem \ref{Problem2} by using triangular decomposition and projection.
Several illustrative examples are given in Section \ref{sec7} to show the applicability and effectiveness of the proposed algorithms.

\section{Triangular decomposition of polynomial systems}\label{sec2}

Here, we recall the basic ideas of Triangular decomposition of polynomial systems. More detailes on the theory and applications of triangular decomposition appear in \cite{Che11,CM12}.

In this section, $\kk$ is a field with algebraic closure $\K$. The notation $R \co \kk[\x]$ indicates the polynomial ring $\kk[x_1, \dotsc, x_n]$ with ordered variables $\x = x_1 < \dotsb < x_n$.
	
Let $p \in \kk[\x] \setminus \kk$. The greates variable of $p$ is called the \emph{main variable} of $p$ and denoted by $\mv(p)$. If $\mv(p) = x_i$ then we can consider $p$ as a univariate polynomial by $x_i$, i.e. $p = \kk[x_1,\ldots,x_{i-1}][x_i]$, and the greatest coefficient of $p$ is called the \emph{initial} of $p$ and denote it by $\I(p)$.

For $F \subset \kk[\x]$, we denote by $\left\langle F\right\rangle$ the ideal in $\kk[\x]$ spanned by $F$, and $V(F)$ the zero set (solution set or algebraic variety) of $F$ in $\K^n$.

Let $I \subset \kk[\x]$ be an ideal.
\begin{itemize}
	\item A polynomial $p \in \kk[\x]$ is called a \emph{zerodivisor modulo $I$} if there exists $q \in \kk[\x]$ such that $pq \in I$ and $p$ or $q$ belongs to $I$. If $p$ is neither 0 nor zerodivisor modulo $I$ then we call $p \in \kk[\x]$ is \emph{regular modulo $I$}
	\item For $h \in \kk[\x]$, the \emph{saturated ideal} of $I$ with respect to (hereafter, w.r.t.) $h$ is an ideal in $\kk[\x]$ as follows:
	\[
		I : h^\infty \co \left\lbrace q \in \kk[\x] \colon \exists m \in \N \text{ such that } h^mq \in I \right\rbrace.
	\]
\end{itemize}

A subset $T \subset \kk[\x] \setminus \kk$ consists of polynomials with pairwise distinct main variables is called a \emph{triangular set}. For triangular set $T \subset \kk[\x]$:
\begin{itemize}
	\item We denote by $\sa(T) \subset \kk[\x]$ the \emph{saturated ideal} of $T$ defined as follows: if $T = \emptyset$ then $\sa(T)$ is the trivial ideal $\{0\}$, otherwise it is the ideal $\left\langle T \right\rangle : h_T^\infty$.
		
	\item Let $p, q \in \kk[\x]$. If either $p$ or $q$ is not constant and has main variable $v$, then we define $\re(p, q, v)$ as the resultant of $p$ and $q$ w.r.t. $v$. We define $\re (p, T)$ inductively as follows: if $T = \emptyset$, then $\re(p,T) = p$; otherwise let $v$ be greatest variable appearing in $T$, then $\re(p,T) = \re \left(\re (p, T_v, v), T_{<v}\right)$.
		
	\item Let $h_T$ be the product of the initials of the polynomials in $T$. The \emph{quasi-component} of $T$ is $W(T) \co V(T) \setminus V(h_T)$. 
\end{itemize}

\begin{definition}[Regular chain]
	A triangular set $T \subset \kk[\x]$ is called a \emph{regular chain} if:
	\begin{enumerate}
		\item $T = \emptyset$; or
		\item $T \setminus \{T_{\max}\}$, where $T_{\max}$ is the polynomial in $T$ with maximum rank, is a regular chain and $\I \left(T_{\max}\right)$ is regular w.r.t. $\sa \left(T \setminus \{T_{\max}\}\right)$.
	\end{enumerate}
\end{definition}

\begin{definition}[Triangular decomposition]
	Let $F \subset \kk[\x]$ be finite. A finite subset $\{T_1, \ldots, T_e\}$ of regular chains of $\kk[\x]$ is called a \emph{triangular decomposition} of $V(F)$ if $V(F) = \bigcup_{i=1}^e W(T_i)$.
\end{definition}

\begin{remark}
	Chen and Maza \cite{CM12} presented the algorithm $\T(F)$ to compute a triangular decomposition of $V(F)$. It has been implemented in MAPLE.
\end{remark}
	
\begin{example}
	Consider the following system:
	\[
		\begin{cases}
			x^2 + y^2 + z^2 & = 4, \\ x^2 + 2y^2 & = 5, \\ xz & = 1.
		\end{cases}
	\]
	Set $F \co \left\lbrace x^2 + y^2 + z^2 - 4, \\ x^2 + 2y^2 - 5, \\ xz - 1\right\rbrace \subset R \co \R[x, y, z]$. Then, $\T(F, R)$ returns three regular chains as follows:
	\[
		\begin{array}{l}
			T_1 = \left\lbrace xz - 1, 2y^2 - 3, x^2 - 2 \right\rbrace, \\
			T_2 =\left\lbrace z + 1, y^2 - 2, x + 1 \right\rbrace, \\
			T_3 = \left\lbrace z-1, y^2-2, x-1 \right\rbrace.
		\end{array}
	\]
	The first regular chain has $h_{T_1} = x$, so $W(T_1) = V(T_1) \setminus V(x)$.  For two remaining ones, we have $W(T_2) = V(T_2)$ and $W(T_3) = V(T_3)$. Since $V(F) = \bigcup_{i=1}^3 W(T_i)$ (note that $V(F) \subset \C^3$), we need to solve three systems as follows:
	\[
		\begin{array}{l l l l l}
			\begin{cases}
				xz -1 & = 0 \\ 2y^2 - 3 & = 0 \\ x^2 - 2 & = 0 \\ x & \neq 0
			\end{cases}, &&
			\begin{cases}
				z & = -1 \\ y^2 & = 2 \\ x & = 1
			\end{cases}, &&
			\begin{cases}
				z & = 1 \\ y^2 & = 2 \\ x & = 1.
			\end{cases}
		\end{array}
	\]
\end{example}
	
\begin{example}
	Consider the following system:
	\[
		\begin{cases}
			x^2 + y + z & = 1, \\ x + y^2 + z & = 1, \\ x + y + z^2 & = 1.
		\end{cases}
	\]
	Due to Cox et al. \cite[Chapter 3, \S1]{CLO15}, the Gr\"obner basis of
	\[
		I = \left\langle x^2 + y + z - 1, x + y^2 + z - 1, x + y + z^2 - 1 \right\rangle
	\]
	w.r.t. lex order reduces to solve the following system:
	\[
		\begin{cases}
			x + y + z^2 - 1 & = 0 \\ y^2 - y - z^2 + z & = 0 \\ 2yz^2 + z^4 - z^2 & = 0 \\ z^6 - 4z^4 + 4z^3 - z^2 & = 0.
		\end{cases}
	\]
	Set $F \co \left\{x^2 + y + z - 1, x + y^2 + z - 1, x + y + z^2 - 1\right\} \subset R \co \R[x, y, z]$. $\T(F,R)$ reduces to solve four systems as follows:
	\[
		\begin{array}{l l l l l ll}
			\begin{cases}
				z - x &= 0 \\ y - x &= 0 \\ x^2 + 2x - 1 &= 0
			\end{cases}, &&
			\begin{cases}
				z &= 0 \\ y &= 0 \\ x - 1 &= 0
			\end{cases}, &&
			\begin{cases}
				z &= 0 \\ y - 1 &= 0 \\ x &= 0
			\end{cases}, &&
			\begin{cases}
				z - 1 &= 0 \\ y &= 0 \\ x &= 0.
			\end{cases}
		\end{array}
	\]
\end{example}
	
\begin{remark}
	A disadvantage of Gr\"obner bases is that they do not necessarily have a triangular set shape. Consequently, solving Gr\"obner bases to construct isomorphims is much harder, and in general, it seems to be impossible.
\end{remark}

\section{Triangular decomposition of semi-algebraic systems}\label{sec3}

In this section, we recall a little bit about triangular decomposition of semi-algebraic systems. For more details, we refer the readers to \cite{Che11,CDMMXX13,CDMXX13} and references therein.

In this section, $\kk$ is a field of characteristic 0 and $\K$ is its algebraic closure. For $p \in \kk[\x] \setminus \kk$, we denote by $\de(p)$ the derivative of $p$ w.r.t. $\mv(p)$.

Let $T \subset \kk[\x]$ be a triangular set. Denote by $\mv(T)$ the set of main variables of the polynomials in $T$. A variable $v \in \x$ is called \emph{algebraic} w.r.t. $T$ if $v \in \mv(T)$, otherwise it is said \emph{free} w.r.t. $T$. We shall denote by $\U = u_1, \dotsc, u_d$ and $\y = y_1, \dotsc, y_m$ respectively the free and the main variables of $T$. We let $d = 0$ whenever $T$ has no free variables.
	
Let $T \subset \kk[\x]$ be a regular chain and $H \subset \kk[\x]$. The pair $[T, H]$ is a \emph{regular system} if each polynomial in $H$ is regular modulo $\sa(T)$. If $H = \{h\}$ then we write $[T, h]$ for short. Regular chain $T$ or regular system $[T, H]$ is \emph{squarefree} if $\de(t)$ is regular w.r.t. $\sa(T)$ for all $t \in T$.
	
Let $[T, H]$ be a squarefree regular system of $\kk[\U, \y]$. Let $\bp$ be the primitive and square free part of the product of all $\re(\de(t), T)$ and all $\re(h, T)$ for $h \in H$ and $t \in T$. We call $\bp$ the \emph{border polynomial} of $[T, H]$.
	
Let us consider four finite subset of $\mathbb{Q}[x_1, \ldots, x_n]$ as follows:
\[
	\begin{array}{l l l l}
		F = \{f_1,\ldots, f_s\}, & N = \{n_1,\ldots, n_t\}, & P = \{p_1, \ldots, p_r\}, & H = \{h_1, \ldots, h_l\}. 
	\end{array}
\]
We denote by $N_\geq$ and $P_>$ the sets of inequalities $\{n_1 \geq 0, \ldots, n_t \geq 0\}$ and $\{p_1 > 0, \ldots, p_r >0\}$, respectively; and $H_{\neq}$ the set of inequations $\{h_1 \neq 0, \ldots, h_l \neq 0\}$.
	
\begin{definition}[Semi-algebraic systems]
	We denote by $\Sf \co [F, N_\geq, P_>, H_{\neq}]$ the \emph{semi-algebraic system} (SAS), that is the conjunction of the following conditions: $f_1 = \cdots = f_s = 0$, $N_\geq$, $P_>$ and $H_{\neq}$.
\end{definition}
	
\begin{definition}[Pre-regular semi-algebraic system]
	Let $[T, P]$ be a squarefree regular system of $\Q[\U, \y]$ with border polynomial $\bp$. Let $B \subset \Q[\U]$ be a polynomial set such that $\bp$ divides the product of polynomials in $B$. We call the triple $[B_{\neq}, T , P_>]$ a \emph{pre-regular semi-algebraic system} (PRSAS) of $\Q[\x]$. Zero set of $[B_{\neq}, T , P_>]$, denoted by $Z_\R (B_{\neq}, T, P_>)$, is the set $(u, y) \in \R^n$ such that $b(u) \neq 0$ for all $b \in B$, $t(u, y) = 0$ for all $t \in T$, and $p(u, y) > 0$ for all $p \in P$.
\end{definition}
	
\begin{lemma}[{\cite[Lemma 1]{CDMMXX13}}]\label{Lem1}
	Let $\Sf$ be a SAS of $\Q[\x]$. Then there exists finitely many PRSASs $[B_{i\neq}, T_i, P_{i>}]$, $i = 1, \ldots, e$, such that $Z_\R (\Sf) = {\mathop \bigcup \limits_{i=1}^e} Z_\R (B_{i\neq}, T_i, P_{i>})$.
\end{lemma}
	
\begin{definition}[Regular semi-algebraic system]
	Let $T \subset \Q[\x]$ be a squarefree regular chain. Let $P \subset \Q[\x]$ be finite and such that each polynomial in $P$ is regular w.r.t. $\sa(T)$, i.e. $[T, P]$ is a regular system. Define $P_> \co \{p > 0 \,|\, p \in P\}$. Let $\Qc$ be a quantifier-free formula over $\Q[\x]$ involving only the $\U$ variables. Let $S = Z_\R(\Qc) \subset \R^d$  be the semi-algebraic subset of $\R^d$ defined by $\Qc$. When $d = 0$, the 0-ary Cartesian product $\R^d$ is treated as a singleton set. We say that $\Rc \co [\Qc, T , P_>]$ is a \emph{regular semi-algebraic system} (RSAS) if:
	\begin{enumerate}
		\item $S$ is a non-empty open subset in $\R^d$,
		\item The regular system $[T, P]$ specializes well at every $u \in S$,
		\item At each $u \in S$, specialized system $[T(u), P(u)_>]$ admits real solutions.
	\end{enumerate}
	The zero set of $\Rc$, denoted by $Z_\R(\Rc)$, is the set of points \fbox{$(u, y) \in \R^d \times \R^{n-d}$} such that $\Qc(u)$ holds, $t(u, y) = 0$ for all $t \in T$ and $p(u, y) > 0$ for all $p \in P$.
\end{definition}
	
\begin{lemma}[{\cite[Lemma 3]{CDMMXX13}}]\label{Lem3}
	Let $[B_{\neq}, T, P_>]$ be a PRSAS of $\Q[\U, \y]$. One can decide whether its zero set is empty or not. If it is not empty, then one can compute a RSAS $[\Qc, T , P_>]$ whose zero set is the same as that of $[B_{\neq}, T, P_>]$.
\end{lemma}
	
\begin{proposition}[{\cite[Theorem 2]{CDMMXX13}}]
	Let $\Sf$ be a SAS of $\Q[\x]$. Then one can compute a (full) triangular decomposition of $\Sf$, that is finitely many RSASs such that the union of their zero sets is the zero set of $\Sf$.
\end{proposition}
	
\begin{remark}
	Chen et al. \cite[Section 7]{CDMMXX13} presented an algorithm to compute a triangular decomposition of a semi-algebraic system $\Sf = [F, N_\geq, P_>, H_{\neq}]$ which was denoted by \textsf{RealTriangularize}$(\Sf)$. It has been implemented to MAPLE.
\end{remark}
	
\begin{example}
	Consider the following system:
	\[
		\begin{cases}
			x^2 + y^2 + z^2 + 2 = 0 \\ 3x^2 + 4y^2 + 4z^2 + 5 = 0.
		\end{cases}
	\]
	Put $F \co \{x^2 + y^2 + z^2 + 2, 3x^2 + 4y^2 + 4z^2 + 5\}$, $N_{\geq} = P_> = H_{\neq} \co \emptyset$. Set $\Sf \co \left[F, \emptyset, \emptyset, \emptyset\right]$ be a SAS in $R \co \R[x, y, z]$. Since $\RT(\Sf,R)$ returns $\emptyset$, the given system has no real root. Note that $\T(F,R)$ returns one regular chain $\left\lbrace z^2 + y^2 - 1, x^2 + 3 \right\rbrace$ which implies that the given system has complex roots.
\end{example}

\section{Projections}\label{sec4}
	
	For many purposes, we need to find solutions of a polynomial system $F$ as well as of a SAS $\Sf$ which consist of $r$ parameters and $n-r$ unknowns. In particular, we want to find out the values of parameters in which the given system admits solutions since it concerns directly with Problem \ref{Problem2}. In this situation, the projections are very useful (see \cite{CDMXX11,CGLMP07}).
	
	\subsection{Polynomial systems}
	
	Let $\kk$ be a field of characteristic zero with algebraic closure $\K$. 
	
	\begin{definition}[Constructible set]
		Let $F = \{f_1,\ldots, f_s\}$ and $H = \{h_1, \ldots, h_l\}$ be two finite set of $\kk[\x]$. The conjunction of $f_1 = \cdots = f_s = 0$ and $h_1 \neq 0, \ldots, h_l \neq 0$ is called a \emph{constructible system} of $\kk[\x]$, and denoted by $[F, H]$. Its zero set in $\K^n$, i.e., $V(F, H) \co V(F) \setminus V(H)$, is called a \emph{basic constructible set} of $\kk[x]$. A \emph{constructible set} of $\kk[\x]$ is a finite union of basic constructible sets of $\kk[\x]$.
	\end{definition}
	
	Now, we set $R \co \kk[\x,\U] = \kk[x_1 > \cdots > x_{n-r} > u_1 > \cdots > u_r]$ be the polynomial ring. For $F \subset R$, we consider $\x$ and $\U$ respectively as unknowns and parameters, that is, its $(x_1, \ldots, x_{n-r})$-solutions are multivariate functions of $(u_1, \ldots, u_r)$. In other words, $R \co \kk[\U][\x] = \kk[u_1 > \cdots > u_r][x_1 > \cdots > x_{n-r}]$. Denote by $\pi^r_\K \colon \K^n = \K^{n-r} \times \K^r \to \K^r$ the projection which maps $(\x, \U)$ in the entire space to $\U$ in the last $r$-dimensional parameter subspace. There are two situations as follows.
	\begin{itemize}
		\item For $F \subset R$, we want to find the image $\pi^r_\K(V(F))$ of its zero set $V(F) \subset \K^n$ under $\pi^r_\K$. 
		
		\item For $F, H \subset R$, we want to find the image $\pi^r_\K(V(F, H))$ of the constructible set $V(F, H) \subset \K^n$ under $\pi^r_\K$.
	\end{itemize}
	
	Note that these images are constructible sets of $\kk[\U]$. In MAPLE, the commands to find $\pi^r_\K(V(F))$ 
	and $\pi^r_\K(V(F, H))$ are $\p(F, r, R)$ and $\p(F, H, r, R)$, respectively.
	
	\begin{example}
		Consider $F \co \{x^2+y^2-1\}$ and $H \co \{x+y-1\}$ in $R \co \R[y, x] \equiv \R[x][y]$, i.e., $y$ and $x$ are respectively the unknown and the parameter. We want to find the values of $x$ such the system $[F, H]$ admits complex roots. To this end, we need to find $\pi^1_\C(V(F, H))$. $\p (F, H, 1, R)$ returns a constructible set of $\R[x]$ as follows:
		\[
			\begin{array}{l l}
				\begin{cases}
					x & \neq 0 \\ x-1 & \neq 0
				\end{cases}, & x = 0.
			\end{array}
		\]
		Thus, the system $[F, H]$ admits complex roots iff $x \neq 1$.
	\end{example}

	\subsection{Semi-algebraic systems}
	
	We set $R \co \R[\x, \U] = \R[x_1 > \cdots > x_{n-r} > u_1 > \cdots > u_r]$ be the polynomial ring, where $\x$ are unknowns and $\U$ are parameters. Denote by $\pi^r_\R \colon \R^n = \R^{n-r} \times \R^r \to \R^r$ the projection which maps $(\x, \U)$ in the entire space to $\U$ in the last $r$-dimensional parameter subspace. For a SAS $\Sf = [F, N, P, H]$ of $R$, we want to find the image $\pi^r_\R(Z_\R(\Sf))$ of its zero set $Z_\R(\Sf) \subset \R^n$ under $\pi^r_\R$.
		
	In this case, the image $\pi^r_\R(Z_\R(\Sf))$ is RSASs of $\R[\U]$ which can be found by the command $\p(F, N, P,  H, r, R)$ in MAPLE.
	
	\begin{example}
		Consider the following quadratic equation:
		\[
			ax^2 + bx + c = 0; \quad a,b,c \in \R, a \neq 0.
		\]
		In this example, we reexamine the well-known result that this quadratic equation admits real roots iff its discriminant $\Delta \co b^2-4ac$ is non-negative. First, we put $F \co \left\lbrace ax^2 + bx+c\right\rbrace$, $N = P \co \emptyset$ and $H \co \{a\}$. Afterwards, we set $\Sf \co [F, N, P, H]$ be a RSAS in $R \co \R[a, b, c][x]$. Then, $\p (F, N, P, H, 3, R)$ returns three RSASs of $\R[a, b, c]$ as follows:
		\[
			\begin{array}{l l l l l}
				\begin{cases}
					b = 0 \\ c = 0 \\ a \neq 0
				\end{cases}, &&
				\begin{cases}
					4ac - b^2 = 0 \\ b \neq 0 \textbf{ and } c \neq 0
				\end{cases}, &&
				b^2 - 4ac > 0 \textbf{ and } a \neq 0.
			\end{array}
		\]
		It is obvious that these results can be sum up by $\Delta \geq 0$.
	\end{example}

\section{Algorithms for Problem \ref{Problem1}}\label{sec5}

Let $L = \s \{X_1, \ldots, X_n\}$ and $L' = \s \{Y_1, \ldots, Y_n\}$ be two $n$-dimensional $\F$-Lie algebras with structure constants $a_{ij}^k$ and $b_{ij}^k$, respectively. An isomorphism $\phi \colon L \to L'$ must satisfy:
\begin{itemize}
	\item $\phi \left(\left[X_i, X_j\right]\right) = \left[\phi(X_i), \phi(X_j)\right]$ for $1 \leq i < j \leq n$,
	\item $\det \left[\phi\right] \neq 0$, where $\left[\phi\right]$ is the matrix of $\phi$.
\end{itemize}
	
Assume that the matrix $\left[\phi\right]$ of $\phi$ is as follows:
\[
	[\phi] =
	\begin{bmatrix}
		z_{11} & \dotsb & z_{n1} \\
		\vdots & \ddots & \vdots \\
		z_{1n}  & \dotsb & z_{nn}
	\end{bmatrix}.
\]
The condition $\phi \left(\left[X_i, X_j\right]\right) = \left[\phi(X_i), \phi(X_j)\right]$ is equivalent to
\[
	\begin{array}{l l}
		\sum \limits_{k=1}^n z_{ks}a_{ij}^k - \sum \limits_{k,l = 1}^n z_{ik}z_{jl}b_{kl}^s = 0, & 1 \leq i < j \leq n, s = 1, \dotsc, n.
	\end{array}
\]
The condition $\det \left[\phi\right] \neq 0$ is equivalent to $1 - z \det [\phi] = 0$, where $z \in \F$ is a new unknown. 
Therefore, the isomorphic conditions to the following system of equations
	\begin{equation}\label{sys-problem1}
		\begin{cases}
			\sum \limits_{k=1}^n z_{ks}a_{ij}^k - \sum \limits_{k,l = 1}^n z_{ik}z_{jl}b_{kl}^s = 0, & 1 \leq i < j \leq n, s = 1, \dotsc, n,\\
			1 - z \det [\phi] = 0,
		\end{cases}
	\end{equation}
	which consists of $\leq n\binom{n}{2} +1$ polynomials in $\F[z, z_{11}, \ldots, z_{1n}, \ldots, z_{n1}, \ldots, z_{nn}]$ of degree $\leq n+1$ with $n^2+1$ unknowns	$z, z_{ij} \in \F$. Hence, we have
	
	\begin{theorem}
		$L$ and $L'$ are isomorphic iff the zero set of \eqref{sys-problem1} is non-empty.
	\end{theorem}
	
	Since \eqref{sys-problem1} is a polynomial system, we can use triangular decompositions. To this end, we first set the polynomial ring $R \co \F[z,  z_{11}, \ldots, z_{1n}, \ldots, z_{n1}, \ldots,  z_{nn}]$ and $F$ be all of polynomials on the left-hand side of \eqref{sys-problem1}. Afterwards, determining whether the zero set of \eqref{sys-problem1} is empty or not is based on the triangular decomposition. If $\F = \C$ then we find a triangular decomposition of $V(F) \subset \C^{n^2+1}$. If $\F = \R$ then we find a triangular decomposition of $Z_\R(F, \emptyset, \emptyset, \emptyset) \subset \R^{n^2+1}$. These procedures are given in Algorithms \ref{Alg1} and \ref{Alg2}.
	
	\begin{algorithm}[h]
		\KwIn{Structure constants $a_{ij}^k \in \C$ of $L$ and $b_{ij}^k \in \C$ of $L'$}
		\KwOut{Yes ($L \cong L'$) or No ($L \ncong L'$)}
			$R \co \C[z, z_{11}, \ldots, z_{1n}, \ldots, z_{n1}, \ldots, z_{nn}]$\;
			$F \co \left\lbrace \text{Polynomials determining equations of \eqref{sys-problem1}}\right\rbrace$\;
			$V \co$ Triangular decomposition of $V(F)$; \qquad \quad /$\mathsf{Triangularize}(F, R)$/\\
			\eIf{$V = \emptyset$}{ouput No}{output Yes}
		\caption{Testing isomorphism of complex Lie algebras}\label{Alg1}
	\end{algorithm}
		
	\begin{algorithm}[h]
		\KwIn{Structure constants $a_{ij}^k \in \R$ of $L$ and $b_{ij}^k \in \R$ of $L'$}
		\KwOut{Yes ($L \cong L'$) or No ($L \ncong L'$)}
		$R \co \R[z, z_{11}, \ldots, z_{1n}, \ldots, z_{n1}, \ldots, z_{nn}]$\;
		$F \co \left\lbrace \text{Polynomials determining equations of \eqref{sys-problem1}}\right\rbrace$\;
		$\Sf \co \left[F, \emptyset, \emptyset, \emptyset\right]$\;
		$Z \co$ Triangular decomposition of $Z_\R(\Sf)$; \qquad /$\mathsf{RealTriangularize}(\Sf, R)$/ \\
		\eIf{$Z = \emptyset$}{ouput No}{output Yes}
		\caption{Testing isomorphism of real Lie algebras}\label{Alg2}
	\end{algorithm}

\section{Algorithms for Problem \ref{Problem2}}\label{sec6}

	Given two $n$-dimensional parametric $\F$-Lie algebras $L(\cb) = \s \{X_1, \ldots, X_n\}$ and $L'(\db) = \s \{Y_1, \ldots, Y_n\}$ whose structure constants are $a_{ij}^k \in \F[\cb]$ and $b_{ij}^k \in \F[\db]$, respectively. Note that two tuples of parameters $\cb$ and $\db$ may satisfy some additional conditions. Our objective is to determine the values of $\cb$ and $\db$ such that $L(\cb) \cong L'(\db)$. We divide into two cases as follows.
	
	\begin{enumerate}[\bf A.]
		\item \fbox{$\F= \C$}. Assume that two tuples $\cb, \db$ satisfy additionally
		\[
			\begin{array}{l l}
				\cb \in \Cs_0 = V(F_0, H_0) \subset \C^r; & F_0, H_0 \in \C[\cb], \\
				\db \in \Cs_1 = V(F_1, H_1) \subset \C^s; & F_1, H_1 \in \C[\db].
			\end{array}
		\]
		Assume that $\phi \colon L \to L'$ is an isomorphism with matrix:
		\[
			[\phi] =
			\begin{bmatrix}
				z_{11} & \dotsb & z_{n1} \\
				\vdots & \ddots & \vdots \\
				z_{1n}  & \dotsb & z_{nn}
			\end{bmatrix}.
		\]
		Since $\cb \in \Cs_0$ and  $\db \in \Cs_1$, the isomorphic conditions is equivalent to
		\begin{equation}\label{sys-problem2}
			\begin{cases}
				\sum \limits_{k=1}^n z_{ks}a_{ij}^k - \sum \limits_{k,l = 1}^n z_{ik}z_{jl}b_{kl}^s = 0, & 1 \leq i < j \leq n, s = 1, \dotsc, n,\\
				1 - z \det [\phi] = 0,\\
				f = 0, & \text{for all $f \in F_0 \cup F_1$}, \\
				h \neq 0, & \text{for all $h \in H_0 \cup H_1$},
			\end{cases}
		\end{equation}
		which consists of $\leq n\binom{n}{2} +1 + |F_0| + |F_1| + |H_0| + |H_1|$ polynomials in
		\[
			\C[z, z_{11}, \ldots, z_{1n}, \ldots, z_{n1}, \ldots, z_{nn}, \cb, \db]
		\]
		of degree $\leq \max \left\lbrace n+1, \max \limits_{f \in F_0 \cup F_1} \deg f, \max \limits_{h \in H_0 \cup H_1} \deg h \right\rbrace$ with $n^2+1$ unknowns $z, z_{ij}$ and $r+s$ parameters $\cb, \db$. Set $F \co \left\lbrace \text{Polynomials determining equations of \eqref{sys-problem2}} \right\rbrace$ and $H \co H_0 \cup H_1$. Then, we have that:	
		
		\begin{theorem}
			$L(\cb)$ and $L'(\db)$ are isomorphic iff $(\cb, \db) \in \pi^{r+s}_\C (V(F, H))$,
			where $\pi_{\mathbb{C}}^{r+s}$ is the projection to the $(\cb,\db)$-parameters space $\C^{r+s}$.
		\end{theorem}
		
		\begin{proof}
			As we have seen above
			\[
				\begin{array}{l l l}
					L'(\cb) \cong L'(\db) & \Leftrightarrow & \text{\eqref{sys-problem2} admits complex roots $(z, z_{ij}, \cb, \db)$} \\
					& \Leftrightarrow & (z, z_{ij}, \cb, \db) \in V(F, H) \\
					& \Leftrightarrow & (\cb, \db) \in \pi^{r+s}_\C (V(F, H)).
				\end{array}
			\]
		\end{proof}
		We can sum up the procedure by Algorithm \ref{Alg3}.
		\begin{algorithm}[h]
			\KwIn{Structure constants $a_{ij}^k \in \C[\cb]$ of $L(\cb)$, $b_{ij}^k \in \C[\db]$ of $L'(\db)$ 
				where $\cb \in \Cs_0 \subset \C^r$, $\db \in \Cs_1 \subset \C^s$ and $\Cs_i = V(F_i, H_i)$}
			\KwOut{$\sqcup_i A_i \subset \Cs \co V(F_0 \cup F_1, H_0 \cup H_1) \subset \C^{r+s}$ s.t. $L(\cb) \cong L(\db)$ iff $\cb, \db \in A_i$ for some $i$}
				$R \co \C[z,z_{11}, \ldots, z_{1n}, \ldots, z_{n1}, \ldots, z_{nn},\cb, \db]$\;
				$F \co \left\lbrace \text{Polynomials determining equations of \eqref{sys-problem2}} \right\rbrace$\;
				$H \co H_0 \cup H_1$\;
				$\pi_\C^{r+s} \left(V(F, H) \right)$; \qquad /$\p (F, H, r+s, R)$/
			\caption{Testing isomorphism of parametric complex Lie algebras}\label{Alg3}
		\end{algorithm}
		
		\item \fbox{$\F = \R$}. Assume that two tuples $\cb, \db$ satisfy additionally
		\[
			\begin{array}{l l}
				\cb \in \Cs_0 = Z_\R(F_0, N_0, P_0, H_0) \subset \R^r; & F_0, N_0, P_0, H_0 \in \R[\cb], \\
				\db \in \Cs_1 = Z_\R(F_1, N_1, P_1, H_1) \subset \R^s; & F_1, N_1, P_1, H_1 \in \R[\db].
			\end{array}
		\]
		Similarly, an isomorphism $\phi \colon L \to L'$ with matrix
		\[
			[\phi] =
			\begin{bmatrix}
				z_{11} & \dotsb & z_{n1} \\
				\vdots & \ddots & \vdots \\
				z_{1n}  & \dotsb & z_{nn}
			\end{bmatrix}.
		\]
		is equivalent to
		\begin{equation}\label{sys-problem2.1}
			\begin{cases}
				\sum \limits_{k=1}^n z_{ks}a_{ij}^k - \sum \limits_{k,l = 1}^n z_{ik}z_{jl}b_{kl}^s = 0, & 1 \leq i < j \leq n, s = 1, \dotsc, n,\\
				1 - z \det [\phi] = 0,\\
				f = 0, & \text{for all $f \in F_0 \cup F_1$}, \\
				n \geq 0, & \text{for all $n \in N_0 \cup N_1$}, \\
				p > 0, & \text{for all $p \in P_0 \cup P_1$}, \\
				h \neq 0, & \text{for all $h \in H_0 \cup H_1$},
			\end{cases}
		\end{equation}
		which consists of $\leq n\binom{n}{2} +1 + |F_0| + |F_1| + |N_0| + |N_1| +  |P_0| + |P_1| + |H_0| + |H_1|$ polynomials in
		$\R[z, z_{11}, \ldots, z_{1n}, \ldots, z_{n1}, \ldots, z_{nn}, \cb, \db]$ of degree
		\[
			\leq \max \left\lbrace n+1, \max \limits_{f \in F_0 \cup F_1} \deg f, \max \limits_{n \in N_0 \cup N_1} \deg n, \max \limits_{p \in P_0 \cup P_1} \deg p, \max \limits_{h \in H_0 \cup H_1} \deg h \right\rbrace
		\]
		with $n^2+1$ unknowns $z, z_{ij}$ and $r+s$ parameters $\cb, \db$. Put
		\[
			\begin{cases}
				F \co \left\lbrace \text{Polynomials determining equations of \eqref{sys-problem2.1}} \right\rbrace \\
				N \co N_0 \cup N_1, \\
				P \co P_0 \cup P_1, \\
				H \co H_0 \cup H_1,
			\end{cases}
		\]
		and set $\Sf \co [F, N, P, H]$ be a SAS of $\R[z, z_{11}, \ldots, z_{1n}, \ldots, z_{n1}, \ldots, z_{nn}, \cb, \db]$. Then, we have that:	
		
		\begin{theorem}
			$L(\cb)$ and $L'(\db)$ are isomorphic iff $(\cb, \db) \in \pi^{r+s}_\R (Z_\R(\Sf))$,
			where $\pi^{r+s}_\R$ is the projection to the $(\cb,\db)$-parameters space $\R^{r+s}$.
		\end{theorem}
		
		We can sum up the procedure by Algorithm \ref{Alg4}.
		
		\begin{algorithm}[h]
			\KwIn{Structure constants $a_{ij}^k \in \R[\cb]$ of $L(\cb)$, $b_{ij}^k \in \R[\db]$ of $L'(\db)$ where $\cb \in \Rs_0 \subset \R^r$, $\db \in \Rs_1 \subset \R^s$, $\Rs_i = Z_\R(F_i, N_i, P_i, H_i)$}
			\KwOut{$\sqcup_i A_i \subset \Rs \co Z_\R(F_0 \cup F_1, N_0 \cup N_1, P_0 \cup P_1, H_0 \cup H_1) \subset \R^{r+s}$ s.t. $L(\cb) \cong L'(\db)$ iff $\cb, \db \in A_i$ for some $i$}
				$R \co \R[z,z_{11}, \ldots, z_{1n}, \ldots, z_{n1}, \ldots, z_{nn},\cb, \db]$\;
				$F \co \left\lbrace \text{Polynomials determining equations of \eqref{sys-problem2.1}} \right\rbrace$\;
				$N \co N_0 \cup N_1$\;
				$P \co P_0 \cup P_1$\;
				$H \co H_0 \cup H_1$\;
				$\Sf \co \left[F, N, P, H\right]$\;
				$\pi_\R^{r+s} \left(Z_\R(\Sf) \right)$; \qquad /$\p (F, N, P, H, r+s, R)$/
			\caption{Testing isomorphism of parametric real Lie algebras}\label{Alg4}
		\end{algorithm}
	\end{enumerate}

\section{Experimentations}\label{sec7}

In this section, we present examples to demonstrate how Algorithms \ref{Alg1}, \ref{Alg2}, \ref{Alg3} and \ref{Alg4} 
can be applied. Let us start with a simple case in dimension 3.

\begin{example}
	Consider  $L = \s \{X_1, X_2, X_3\}$ and $L' = \s \{Y_1, Y_2, Y_3\}$ with
	\[
		\begin{array}{l l}
			[X_1, X_3] = 2X_1 + X_2, & [X_2, X_3] = -X_1 + 2X_2, \\
			\left[Y_1, Y_3\right] = 3Y_1 + 2Y_2, & [Y_2, Y_3] = -Y_1 + Y_2.
		\end{array}
	\]
	Assume that $\phi \colon L \to L'$ is an isomorphism with matrix
	\[
		[\phi] =
			\begin{bmatrix}
				z_{11} & z_{21} & z_{31} \\ z_{12} & z_{22} & z_{32} \\ z_{13} & z_{23} & z_{33}
			\end{bmatrix}.
	\]
	By simple computations, system \eqref{sys-problem1} consists of 9 equations as follows:
	\[
		\begin{cases}
			f_1 \co 1 - z \det [\phi] & = 0 \\
			f_2 \co -z_{13} + 2z_{23} &= 0 \\ 
			f_3 \co 2z_{13} + z_{23} & = 0 \\
			f_4 \co -3z_{11}z_{23} + z_{12}z_{23} + 3z_{13}z_{21} - z_{13}z_{22} & = 0 \\
			f_5 \co -2z_{11}z_{23} - z_{12}z_{23} + 2z_{13}z_{21} + z_{13}z_{22} & = 0 \\
			f_6 \co -3z_{11}z_{33} + z_{12}z_{33} + 3z_{13}z_{31} - z_{13}z_{32} + 2z_{11} + z_{21} & = 0 \\
			f_7 \co -2z_{11}z_{33} - z_{12}z_{33} + 2z_{13}z_{31} + z_{13}z_{32} + 2z_{12} + z_{22} & = 0 \\
			f_8 \co  -3z_{21}z_{33} + z_{22}z_{33} + 3z_{23}z_{31} - z_{23}z_{32} - z_{11} + 2z_{21} & = 0 \\
			f_9 \co-2z_{21}z_{33} - z_{22}z_{33} + 2z_{23}z_{31} + z_{23}z_{32} - z_{12} + 2z_{22} & = 0
		\end{cases}
	\]
	First of all, we put $F \co \left\lbrace f_1, f_2, \ldots, f_9 \right\rbrace$. Afterwards, we set $\Sf \co [F, \emptyset, \emptyset, \emptyset]$ 
	be a SAS in $\R[z, z_{11}, z_{12}, z_{13}, z_{21}, z_{22}, z_{23}, z_{31}, z_{32}, z_{33}]$. 
	Then, $\RT (\Sf, R)$ returns two RSASs as follows:
	\[
		\begin{array}{l l l}
			\begin{cases}
				\left(2z_{21}^2 - 2z_{22}z_{21} + z_{22}^2\right) z - 1 & = 0 \\
				z_{11} + z_{21} - z_{22} & = 0 \\
				z_{12} + 2z_{21} - z_{22} & =0 \\
				z_{13} & = 0 \\
				z_{23} & = 0 \\
				z_{33} - 1 & = 0 \\
				z_{22} & \neq 0
			\end{cases}, &&
			\begin{cases}
				2z_{21}^2 z - 1 & = 0 \\
				z_{11} + z_{21} & = 0 \\
				z_{12} + 2z_{21} & = 0 \\
				z_{13} & = 0 \\
				z_{22} & = 0 \\
				z_{23} & = 0 \\
				z_{33} - 1 &= 0 \\
				z_{21} & \neq 0
			\end{cases}
		\end{array}
	\]
	Since the output is non-empty, $L$ and $L'$ are isomorphic over $\R$. So are they over $\C$. 
	From these RSASs, an isomorphism $\phi \colon L \to L'$ can be easily constructed, namely, we have
	\[
		[\phi] =
			\begin{bmatrix}
			-1 & 1 & 0 \\ 
			-2 & 0 & 0\\ 
			0 & 0 & 1
		\end{bmatrix}.
	\]
\end{example}

\begin{example}
	Consider two 4-dimensional real Lie algebras given in \cite{Mub63a}:
	\[
		\begin{array}{l l l l}
			g_{4,8} \colon &[e_2, e_3] = e_1, & [e_2, e_4] = e_2, & [e_3, e_4] = -e_3; \\
			g_{4,9} \colon & [e_2, e_3] = e_1, & [e_2, e_4] = -e_3, & [e_3, e_4] = e_2.
		\end{array}
	\]	
	System \eqref{sys-problem1} determining an isomorphism $\phi \colon g_{4,8} \to g_{4,9}$ consists of 23 equations $f_1 = \cdots = f_{22} = 0$. We first put $F \co \{f_1, \ldots, f_{22}\}$ and then set $\Sf \co [F, \emptyset, \emptyset, \emptyset]$ be a SAS in $R \co \R[z, z_{11}, \ldots, z_{14}, \ldots, z_{41}, \ldots, z_{44}]$. Then, $\RT(\Sf, R)$ returns $\emptyset$, i.e., $Z_\R(\Sf) = \emptyset$, and thus $g_{4,8} \ncong g_{4,9}$ over $\R$. Note that if we consider $g_{4,8}$ and $g_{4,9}$ over $\C$ then the polynomial ring is $R \co \C[z, z_{11}, \ldots, z_{14}, \ldots, z_{41}, \ldots, z_{44}]$. In this case, $\T(F, R)$ returns one regular chain which reduces to the following system:
	\[
		\begin{cases}
				4z_{23}^2z_{33}^2z_{44}z + 1 & = 0 \\
				z_{11} - 2z_{23}z_{33}z_{44} & = 0 \\
				z_{12}, z_{13}, z_{14} & = 0 \\
				z_{21} + (z_{42} - z_{43}z_{44})z_{23} & = 0 \\
				z_{22} - z_{23}z_{44} & = 0\\
				z_{24} & = 0 \\
				z_{31} - \left(z_{42} + z_{43}z_{44}\right)z_{33} & = 0 \\ 
				z_{32} + z_{33}z_{44} & = 0  \\
				z_{34} & = 0 \\
				z_{44}^2 + 1 & = 0 \\
				4z_{23}^2z_{33}^2z_{44} & \neq 0
		\end{cases}
	\]
	and thus $g_{4,8} \cong g_{4,9}$ over $\C$. Solving this system gives us an isomorphism as follows:
	\[
		[\phi] =
		\begin{bmatrix}
			2i & 0 & 0 & 0 \\ 
			0 & i & -i & 0 \\ 
			0 & 1 & 1 & 0 \\ 
			0 & 0 & 0 & i
		\end{bmatrix}; \quad \left(\text{$i$ is the imaginary unit}\right).
	\]	
\end{example}

\begin{example}
	In this example, we will optimize parameters $\beta$ and $\gamma$ of the 5-dimensional real Lie algebras $g_{5,9}^{\beta\gamma}$
	given in \cite{Mub63b}:
	\[
		\begin{array}{l l l l l}
			[e_1, e_5] = e_1, & [e_2, e_5] = e_1 + e_2, & [e_3, e_5] = \beta e_3, & [e_4, e_5] = \gamma e_4 & \left(\beta \gamma \neq 0\right).
		\end{array}
	\]
	Our objective is equivalent to find out the conditions of two pairs of real numbers $(\beta, \gamma)$ and $(\delta, \sigma)$ for which $g_{5,9}^{\beta\gamma}$ and $g_{5,9}^{\delta\sigma}$ are isomorphic. Now, system \eqref{sys-problem2.1} consists of 45 equations $f_1 = \cdots = f_{45} = 0$. First,  we put $F \co \{f_1, \ldots, f_{45}\}$, $N = P \co \emptyset$ and $H \co \{\beta, \gamma, \delta, \sigma\}$. Next, we set $\Sf \co [F, N, P, H]$ be a SAS in $R \co \R[z, z_{11}, \ldots, z_{1n}, \ldots, z_{51}, \ldots, z_{55}, \beta, \gamma, \delta, \sigma]$. Then, $\p(F, N, P, H, 4, R)$ returns two RSASs as follows:
	\[
		\begin{array}{l l l}
			\begin{cases}
				\beta - \delta = 0 \\ \gamma - \sigma = 0 \\ \delta \neq 0 \textbf{ and } \sigma \neq 0
			\end{cases}, &
			\begin{cases}
				\beta - \sigma = 0 \\ \gamma - \delta = 0 \\ \delta \neq 0 \textbf{ and } \sigma \neq 0.
			\end{cases}
		\end{array}
	\]
	The first RSAS is trivial since two Lie algebras coincide. The second one indicates that two pairs $(\beta, \gamma)$ and $(\gamma, \beta)$ are equivalent, i.e., $g_{5,9}^{\beta\gamma} \cong g_{5,9}^{\gamma\beta}$. Since $(\beta, \gamma)$ and $(\gamma, \beta)$ are symmetric over the line $y=x$ in the punctured $Oxy$-plane (see Figure \ref{fig-g59}),
	\begin{figure}[!h]
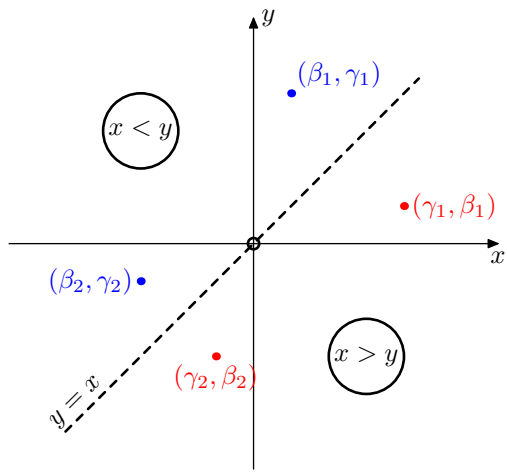

		\centering \parbox{10cm}{\convertMPtoPDF{g59.1}{1}{1}}
		\caption{Pairs of real numbers $(\beta, \gamma)$ and $(\gamma, \beta)$.}\label{fig-g59}
	\end{figure}
	we can choose the pair $(\beta, \gamma)$ below the line $y=x$ including this line except the origin. To sum up, the desired optimal conditions for parameters of $g_{5,9}^{\beta\gamma}$ is $\beta \geq \gamma$ and $\beta\gamma \neq 0$.
\end{example}

\begin{example}
	Consider two 6-dimensional Lie algebras with basis $\{e_1, \dotsc, e_6\}$ as follows:
	\[
		\begin{array}{l l l l l l}
			L_{6,8} \colon & [e_1, e_2] = e_3+e_5, & [e_1, e_3] = e_4, & [e_2, e_5] = e_6, \\
			L_{6,10}^c \colon & [e_1, e_2] = e_3, & [e_1, e_3] = e_5, & [e_1, e_4] = e_6,& [e_2, e_4] = e_5, & [e_2, e_3] = ce_6,
		\end{array}
	\]
	where $c$ is a non-zero parameter. These are two families of 6-dimensional nilpotent Lie algebras over a field of characteristic zero which were classified by Morozov \cite[Section 2]{Mor58}. In this example, we reexamine Morozov's results, i.e., we want to check if $L_{6,8}$ and $L_{6,10}^c$ belong to two non-isomorphic classes. To this end, we will find out all values of parameter $c$ for which $L_{6,8} \cong L_{6,10}^c$.
	
	In this case, both two systems \eqref{sys-problem2} and \ref{sys-problem2.1} consist of 55 equations $f_1 = \cdots = f_{55} = 0$. Set $F \co \{f_1, \ldots, f_{55}\}$, $N = P \co \emptyset$ and $H \co \{c\}$.
	\begin{itemize}
		\item Over $\R$, we set $R \co \R[z, z_{11}, \ldots, z_{1n}, \ldots, z_{61}, \ldots, z_{66}, c]$. $\p(F, N, P, H, 1, R)$ returns one RSAS which is $c>0$. This means that all real Lie algebras $L_{6,10}^c$ with $c>0$ are isomorphic to $L_{6,8}$.
		
		\item Over $\C$, we set $R \co \C[z, z_{11}, \ldots, z_{1n}, \ldots, z_{61}, \ldots, z_{66}, c]$. $\p(F, H, 1, R)$ returns one constructible set which is $c \neq 0$. This means that all complex Lie algebras $L_{6,10}^c$ are isomorphic to $L_{6,8}$.
	 \end{itemize}
	 To sum up, Morozov's classification of 6-dimensional Lie algebras over a field of characteristic zero is redundant, and we can refine his results appropriately.
\end{example}

\bibliographystyle{amsplain}

\end{document}